\documentclass[12pt,a4paper]{article}
\usepackage{amsmath,amssymb,amsthm}
\usepackage{graphicx}
\usepackage[hidelinks]{hyperref}
\usepackage{color,accents}
\ifdefined\pdfminorversion \pdfminorversion=7 \fi

\theoremstyle{plain}
\newtheorem{theorem}{Theorem}

\newtheorem{lemma}[theorem]{Lemma}

\theoremstyle{definition}

\newtheorem{conjecture}[theorem]{Conjecture}
\theoremstyle{remark}

\normalsize
\def\nicebreak{\vskip0pt plus50pt\penalty-300\vskip0pt plus-50pt }

\newcommand{\cyc}{\varPhi}

\newcommand{\llbrace}{\{\mkern-3mu\{}
\newcommand{\rrbrace}{\}\mkern-3mu\}}
\newcommand{\calM}{\mathcal{M}}
\newcommand{\calA}{\mathcal{A}}
\newcommand{\calR}{\mathcal{R}}
\newcommand{\calL}{\mathcal{L}}
\newcommand{\abs}[1]{\mathopen|#1\mathclose|}
\newcommand{\Sd}{\varSigma_d}
\newcommand{\redto}{\rightarrowtail}
\renewcommand{\dfrac}[2]{\lower0.12ex\hbox{\large$\textstyle\frac{#1}{#2}$}}
\newcommand{\seqnum}[1]{\href{https://oeis.org/#1}{\rm \underline{#1}}}

\title{Paths through Equally Spaced Points on a Circle}

\author{Brendan D. McKay\thanks{This project employed resources
  from the National Computational Infrastructure of Australia.}\\
School of Computing\\
Australian National University\\
Canberra, ACT 2601, Australia\\
\texttt{brendan.mckay@anu.edu.au}
\and
Tim Peters\\
Python Software Foundation\\
USA\\
\texttt{tim@python.org}
}
\date{} 

\begin{document}
\maketitle

\begin{abstract}
Consider $n$ points evenly spaced on a circle, and a
path of $n-1$ chords that uses each point once.
There are $m=\lfloor n/2\rfloor$ possible chord lengths, so
the path defines a multiset of $n-1$ elements drawn
from $\{1,2,\ldots,m\}$.
The first problem we consider is to characterize the multisets which are
realized by some path.
Buratti conjectured that all multisets can be realized when~$n$
is prime, and a generalized conjecture for all~$n$ was
proposed by Horak and Rosa.  Previously the conjecture was proved
for $n\leq 19$ and $n=23$; we extend this to $n\leq 37$
(OEIS sequence A352568).

The second problem is to determine the number of distinct
(euclidean) path lengths that can be realized.  For this
there is no conjecture; we extend current knowledge from
$n\leq16$ to $n\leq37$ (OEIS sequence A030077).
When $n$ is prime, twice a prime, or a power of~2, we prove
that two paths have the same length only if
they have the same multiset of chord lengths.
\end{abstract}

\section{Introduction}

Consider $n$ points equally spaced around a circle.
There are $m=\lfloor n/2\rfloor$ possible chord lengths.
The \textit{type} of a chord is its position in the list
of chord lengths in increasing order; thus
a chord of type~1 is between two adjacent points and
a chord of type~$m$ is between two points as antipodal
as possible.
If the points are numbered cyclically, the type of the chord
between points~$i$ and~$j$ is $\min\{\abs{i-j},n-\abs{i-j}\}$.

Now connect the points by a polygonal path using each point exactly once.
The \textit{associated multiset} of the path is the multiset of the types of the
chords.
We consider two questions:

\nicebreak
\noindent
(Q1) Which multisets are the associated multiset of some path?\\[1ex]
(Q2) How many distinct (euclidean) lengths can paths have?

We denote a multiset by the notation
$[\ell_1,\ldots,\ell_m]$, where $\ell_j$ is the number of elements
equal to~$j$.
Figure~\ref{circle} shows the associated multiset of a path
in this notation.

\begin{figure}[t]
\[  \includegraphics[scale=1]{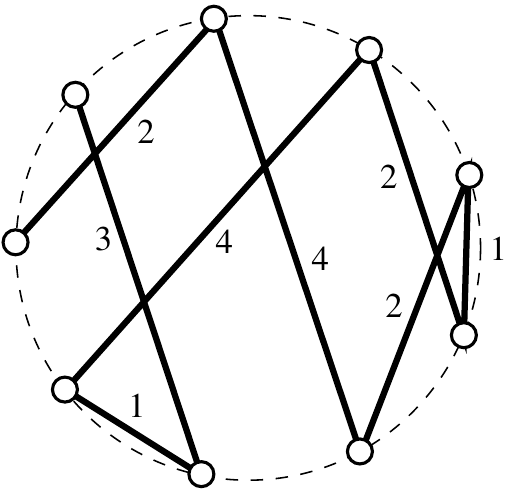} \]
   \caption{A path for $n=9$ with associated multiset $[2,3,1,2]$.\label{circle}}
\end{figure}

Three classes of multisets are relevant to this study.
\begin{itemize}\itemsep=0pt
  \item[(a)] $\calM_n$ is the class of all multisets
     $[\ell_1,\ldots,\ell_m]$ such that $m=\lfloor n/2\rfloor$ and
     $\sum_{j=1}^m \ell_j=n-1$.
  \item[(b)] The \textit{admissible} multisets are the class
     $\calA_n\subseteq \calM_n$ of multisets with this additional
     property:  for each divisor $d$ of $n$,
     $\sum_{j=1}^{\lfloor m/d\rfloor} \ell_{jd} \leq n-d$.
  \item[(c)] The \textit{realizable} multisets are the class
    $\calR_n\subseteq \calM_n$ of multisets associated
    with some path.
\end{itemize}

In 2007, Marco Buratti communicated to Alex Rosa the conjecture
that $\calR_n=\calM_n$ if $n$ is prime~\cite{HorakRosa}.
Despite its simple statement, the conjecture remains open,
though Mariusz Meszka confirmed it by computer for $n\leq 23$~\cite{Meszka}.
It is easy to see that the primality of $n$ is essential for
$\calR_n=\calM_n$, however Horak and Rosa proposed a more
general conjecture that has drawn a lot of attention~\cite{HorakRosa}.

\begin{conjecture}[Buratti--Horak--Rosa]\label{BHR}
 $\calR_n=\calA_n$ for $n\geq 1$.
\end{conjecture}

Horak and Rosa noted that $\calR_n\subseteq\calA_n$; for a 
self-contained proof see Pasotti and Pellegrini~\cite{PasottiPellegrini1}.
Meszka confirmed the conjecture for $n\leq 18$~\cite{Meszka}.
In addition, Conjecture~\ref{BHR} has been proved for a considerable
number of special cases~\cite{Capparelli,Chand,Ollis,OllisGrow,
PasottiPellegrini1,PasottiPellegrini2,Vazquez}.
We will prove:

\begin{theorem}\label{mainthm}
  The Buratti--Horak--Rosa conjecture is true for $n\leq 37$.
\end{theorem}

For question Q2, the first investigation we are aware of was carried out
in the mid-1980s by Daniel Gittelson, then at the University of Michigan
School of Medicine.  Gittelson found the counts up to 12~points~\cite{DLG}.
T.\,E.~Noe added the counts up to 16~points in 2007~\cite{Noe}.
We will continue the sequence up to $n=37$.

\nicebreak
\section{Realization of multisets}

\begin{figure}[ht]
\[  \includegraphics[scale=0.8]{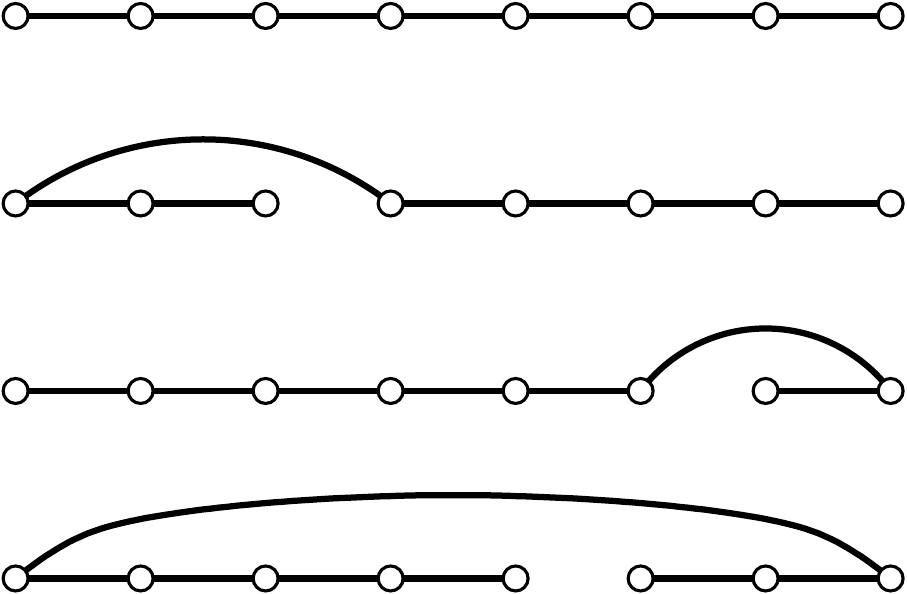} \]
\caption{A path and three types of modification\label{pathflip}}
\end{figure}

Our most computationally challenging task was to find paths that
realize each of approximately $6.4\times 10^{13}$ admissible multisets.
For this
a simple backtrack search is by far not efficient enough for large~$n$,
so we designed several improved algorithms.
Here we describe the two most successful.
Note that, although many special cases of Conjecture~\ref{BHR} have
been proved, they are only a small fraction of cases for large~$n$,
so we chose to not exclude them from our search.

One observation used by both methods is this: if $k$ is an 
integer coprime to $n$, then $kM$ is realizable if and only
if $M$ is realizable, where
$kM = \llbrace k\ell \bmod n \mid \ell \in M\rrbrace$.
Thus, only one of the multisets in each equivalence class
defined by this congruence need be tested.

One approach was a randomized form of hill-climbing.
Figure~\ref{pathflip} shows three ways to transform a path,
which were employed for theoretical purposes by Horak and Rosa~\cite{HorakRosa}.
In each case, the induced multiset loses one element and
gains another (perhaps equal).
The idea is to start with some path and then repeatedly apply
transformations until the required multiset is achieved.

Choice of transformation was made at random with a strong
bias towards beneficial moves.  Transformations which moved
away from the target (fewer chords matched the required
multiset) were given a weight of 1, sideways transformations
(same number of matches) a weight of 100, and transformations
that moved closer to the target had a weight of 10000 (or $\infty$ if the
target multiset was immediately reached).  The admissible
multisets were processed in lexicographic order, meaning that
each multiset was usually very similar to the one before.
This meant it was efficient to use the solution for each 
multiset as the starting point to search for a solution for the
following multiset.

There was a large limit on the number of iterations, with
code to start over with a random path if the limit was reached,
but this never happened.
As an example, for $n=34$ the average number of iterations
was~$104$.

The second method for realizing multisets was a mixture of
random and deterministic search.
A  boolean array indexed by a  multiset ranking function
kept track of which multisets had been realized, while simultaneously
one process generated random paths and another realized
multisets using a backtracking search. In both cases, multisets
related by coprime multiplication (as described above) and by the last
operation in Figure~\ref{pathflip} were also marked off. 
The backtracking search had some problem-specific features
that we now describe.

At each recursion level, we have a path so far, and a multiset of
chord types that still need to be used. For each distinct chord
type remaining, there can be 0, 1 or 2 unused points that can
reached by such a chord.  The order in which the possibilities
are attempted is important for the average efficiency.
When all possibilities are exhausted, backtrack to the previous
level occurs.

Heuristics are used to try to guess at a good order in which to try
chord types. In general, the program favors a pair of chord type and
next point that leaves the next point with the fewest number of possible
exits,
and also favors chord types of which the fewest remain to be used.
This all has much in common with the usual heuristics in backtracking
Hamiltonian path solvers, including various conditions that allow to
prune a search ``early''.

There are also some specializations, driven by experience. For
example, if $n$ is even, and only one instance of an odd chord type
remains, there is only one possible place that chord can appear in the
remaining path.

This usually worked very well, but in a small percentage of cases
would take hundreds of times longer. A pleasant surprise was that
Limited Discrepancy Search (LDS)~\cite{LDS}, adapted for non-binary
trees, proved extremely effective, 99.9\%\ of the time finding a
path with discrepancy no larger than 1, and with
discrepancy 2 in 99\%\ of the remaining cases.
However, particularly for the largest size $n=28$ completed by this
method, a handful of cases required discrepancies as high as 14
and took minutes of cpu time each.

For both implementations, whenever a realization is found it is
checked in separate code. 
The result of the computations was that all admissible multisets
for $n\leq 37$ are realizable.
All cases for $n\leq 28$ were completed with both methods.

\section{When two paths have the same length}\label{identities}

For definiteness we will assume a circle of radius~1.
The length of a chord of type~$j$ is $2\sin(j\pi/n)$.
Therefore, realizable multisets $[\ell_1,\ldots,\ell_m]$
and $[\ell'_1,\ldots,\ell'_m]$ have the same length if and
only if $\sum_{j=1}^m (\ell'_j-\ell_j) \sin(j\pi/n)=0$.
Also note that $\sum_{j=1}^m (\ell'_j-\ell_j)=0$, since
all multisets in $\calM_n$ have $n-1$ elements.

We will call a sequence $(a_1,\ldots,a_m)$ of rational
numbers an \textit{identity} if
\begin{align}
       \sum_{j=1}^m a_j\sin\Bigl(\frac {j\pi}{n}\Bigr) &= 0, \
            \text{~~and} \label{eq1} \\[-0.8ex]
       \sum_{j=1}^m a_j &= 0. \label{eq0}
\end{align}
Let $z=e^{i\pi/n}$, which is a primitive $(2n)$-th root of~1.
Then $\sin\bigl( \frac  {j\pi}{n}\bigr) = \dfrac{1}{2i}(z^j-z^{-j})$.
Thus~\eqref{eq1} can be written
\[
    \frac{1}{2i} \sum_{j=1}^m a_j (z^j-z^{-j}) = 0.
\]
Since $z\ne 0$, this is equivalent to $P_n(z)=0$, where
\begin{equation}\label{eq3}
    P_n(z) = z^m \sum_{j=1}^m a_j (z^j-z^{-j})
     = \sum_{j=1}^{m} a_j z^{m+j} - \sum_{j=1}^m a_j z^{m-j}.
\end{equation}
Note that $P_n(z)$ is a polynomial with rational coefficients.

The \textit{cyclotomic polynomial} of order $2n$ is the
monic polynomial $\cyc_{2n}(x)$ whose zeros are the primitive
$(2n)$-th roots of unity.  In particular, $\cyc_{2n}(z)=0$.
For the theory of cyclotomic polynomials, see Prasolov~\cite[pp.\,89--99]{Prasolov}.
We will require these properties: (1) up to scaling, $\cyc_{2n}(z)$ is the
unique nonzero rational polynomial of least degree that has~$z$ as a zero;
(2) the degree of $\cyc_{2n}(x)$ is Euler's totient function $\varphi(2n)$
(the number of positive integers less than $2n$ and coprime to $2n$);
(3) $\cyc_{2n}(x)$ is palindromic (the list of coefficients reads the same
forwards and backwards).

Perform a rational polynomial division:
\[
      P_n(x) = C_n(x)\cyc_{2n}(x) + R_n(x),
\]
where $C_n(x)$ is a rational polynomial and $R_n(x)$ has lower degree than
$\cyc_{2n}(x)$.
Since $R_n(z)=0$, the minimality of $\cyc_{2n}(x)$ implies
 that $R_n(x)$ is identically zero.

The coefficients of $R_n(x)$ are linear combinations of $a_1,\ldots,a_m$
which must equal~0.
Including equation~\eqref{eq0}, we have a linear system whose solution
space is the vector space of all identities.

\nicebreak
\subsection{Example}

Consider $n=15$, $m=7$.
The cyclotomic polynomial is
\[
    \cyc_{30}(x) = x^8 + x^7 - x^5 - x^4 - x^3 + x + 1.
\]
Performing the division, we find $P_{15}(x) = C_{15}(x) \cyc_{30}(x) + R_{15}(x)$,
where
\begin{align*}
  C_{15}(x) &= a_7x^6 + (a_6-a_7)x^5 + (a_5-a_6+a_7)x^4 
     + (a_4-a_5+a_6)x^3  \\
     &{\qquad} + (a_3-a_4+a_5)x^2
     + (a_2-a_3+a_4+a_7)x + a_1-a_2+a_3+a_6-a_7, \\
  R_{15}(x) &= (-a_1+a_2+a_5-a_6+a_7)x^7 + (-a_1+a_2+a_4+a_7) x^6
    + (a_1-a_2+a_3+a_6) x^5 \\
    &{\qquad} + (a_1-a_3+a_6-a_7) x^4 + (a_1-a_2-a_4-a_7) x^3
    + (-a_2-2a_5-a_7)x^2 \\
    &{\qquad} + (-a_1-a_4-2a_6) x - a_1+a_2-a_3-a_6.
\end{align*}
Now we require $R_{15}(x)=0$ identically, so we can set each of
the coefficients to 0 and we also need $a_1+a_2+a_3+a_4+a_5+a_6+a_7=0$.
In matrix form:

\[
\begin{bmatrix}
-1 & 1 & -1 & 0 & 0 & -1 & 0 
\\
 -1 & 0 & 0 & -1 & 0 & -2 & 0 
\\
 0 & -1 & 0 & 0 & -2 & 0 & -1 
\\
 1 & -1 & 0 & -1 & 0 & 0 & -1 
\\
 1 & 0 & -1 & 0 & 0 & 1 & -1 
\\
 1 & -1 & 1 & 0 & 0 & 1 & 0 
\\
 -1 & 1 & 0 & 1 & 0 & 0 & 1 
\\
 -1 & 1 & 0 & 0 & 1 & -1 & 1 
\\
 1 & 1 & 1 & 1 & 1 & 1 & 1 
\end{bmatrix}\;
\begin{bmatrix}
  a_1\\a_2\\a_3\\a_4\\a_4\\a_6\\a_7
\end{bmatrix}
=
\begin{bmatrix}
 0\\0\\0\\0\\0\\0\\0\\0\\0
\end{bmatrix}
.
\]
The solution space has dimension 2:
\[
    \bigl\langle
          (1, 0, -1, -1, -1, 0, 2), (0, 1,  0, -2, -1, 1, 1)
    \bigr\rangle.
\]

\subsection{What is the dimension?}

We now determine the dimension of the vector space of identities.
For those values of~$n$ where the dimension is~0, only
paths with the same multiset of chord types have the same length.

 \begin{theorem}\label{dimen0}
     For all $n\geq 1$, the dimension of the vector space of identities is
   \[
        \max\bigl\{ 0, m - \tfrac12\varphi(2n) - 1\bigr\}.
   \]
   In particular, the dimension is~0 if
  and only if $n=9$, or $n$ is a prime, twice a prime, or a power of~2.
\end{theorem}
\begin{proof}
   For a polynomial $f(x)=\sum_{j=0}^k b_j x^j$, we say that $f(x)$
   is \textit{$k$-palindromic} if $b_{k-j}=b_j$ for all~$j$, and
   \textit{$k$-antipalindromic} if $b_{k-j}=-b_j$ for all~$j$.
   These properties are respectively equivalent to
   $x^kf(1/x)=f(x)$ and $x^kf(1/x)=-f(x)$.
   As examples, $\cyc_{2n}(x)$ is $\varphi(2n)$-palindromic, 
   while $P_n(x)$ defined in~\eqref{eq3} is $2m$-antipalindromic.

   Consider the equation $P_n(x)=C_n(x)\cyc_{2n}(x)$.  The degree
   of $C_n(x)$ is at most $t=2m-\varphi(2n)$.  Note that
   $\varphi(2n)$ is even, so $t$ is also even.  Also,
    \[
       x^t C_n(1/x) = \frac{x^{2m} P_n(1/x)}{x^{\varphi(2n)}\cyc_{2n}(1/x)}
      = \frac{-P_n(x)}{\cyc_{2n}(x)} = -C_n(x),
    \]
   so $C_n(x)$ is $t$-antipalindromic.  By the same logic, if
   $C_n(x)$ is $t$-antipalindromic then $P_n(x)$ is $2m$-antipalindromic
   and so corresponds to a solution of~\eqref{eq1}.
   
   Choosing a basis of $t/2$ linearly independent  $t$-antipalindromic 
   polynomials for $C_n(x)$, such as $x^j-x^{t-j}$ for $0\leq j\leq \frac12 t-1$,
   we find that the vector space of solutions of~\eqref{eq1} has
   dimension~$t/2$.
   If that vector space lies within the hyperplane defined by~\eqref{eq0},
   the vector space of identities has dimension $t/2$; otherwise it
   has dimension~$t/2-1$.
    
   Recall that $\varphi(2n) = n\prod_p\,(1-1/p)$ where the
   product is over all distinct odd primes~$p$ dividing~$n$.
   From this, a little calculation shows that $t=0$ only 
   if $n$ is an odd prime ($\varphi(2n)=n-1$) or a power of~2
   ($\varphi(2n)=n$).
   
   To show that the dimension is $t/2-1$ rather  
   than $t/2$ when $t\geq 2$, we have only to find $(a_1,\ldots,a_m)$ that
   satisfies~\eqref{eq1} but not~\eqref{eq0}.
   Let's call this an \textit{improper identity}.
   
   Note that if $(a_1,\ldots,a_{\lfloor n/2\rfloor})$ is an improper
   identity for $n$ then
   $(a'_1,\ldots, a'_{\lfloor kn/2\rfloor})$ is an improper
   identity for $kn$, where
   $a'_{kj}=a_j$ for $1\leq j\leq m$ and $a'_{kj}=0$ otherwise.
   Therefore, it suffices to find improper identities for some
   values of $n$ that divide any value of $n$ giving $t\geq 2$.
   The minimum set is: twice an odd prime, the square of an odd prime,
   and the product of two distinct odd primes. 
   
   First, suppose that $n$ is twice an odd prime.  Then
   $\cyc_{2n}(x)=\sum_{j=0}^{n-1} (-1)^j x^{2j}$ and $t=2$.
   Taking $C_n(x)=x^2-1$,
   notice that the coefficients of $C_n(x)\cyc_{2n}(x)$ are
   all $\pm 2$ except for the first and last which are $\pm 1$.
   Therefore, condition~\eqref{eq0} is not satisfied
   and we have an improper identity.
   
   Next suppose that $n=p^2$ where $p$ is an odd prime.
   Then $\cyc_{2n}(x) = \sum_{j=0}^{p-1} x^{jp}$ and
   $t=p-1$.  Consider $C_n(x)=x^{t/2-1}-x^{t/2+1}$,
   so $C_n(x)\cyc_{2n}(x) = \sum_{j=0}^{p-1} \,(x^{jp+t/2-1}-x^{jp+t/2+1})$.
   The coefficients are thus in $\pm 1$ pairs, but for $j=(p-1)/2$ the
   pair is $x^{m-1}-x^{m+1}$.  Thus, $\sum_{j=0}^m a_j$, which is
   the sum of the coefficients up to and including that of $x^{m-1}$,
   equals~1 and condition~\eqref{eq0} is violated. So this is
   an improper identity.
   
   Finally, consider $n=pq$ where $3\leq p<q$ are primes.  Then
   $t = p+q-2$ and
   \[ 
       \cyc_{2n}(x) = \frac{(x+1)(x^{pq}+1)}{(x^p+1)(x^q+1)}.
   \]
  Consider the $t$-antipalindromic polynomial
  $C_n(x) = x^{(q-3)/2}(x-1)(x^p+1)$.
   Then
   \[
        C_n(x)  \cyc_{2n}(x) = \frac{x^{(q-3)/2} (x^2-1)(x^{pq}+1)}{x^q+1}
           = x^{(q-3)/2} (x^2-1)(x^{pq}+1) \sum_{j\geq 0}\, (-1)^j x^{jq}.
   \]
   Since we are only interested in the coefficients up to $x^{m-1}$,
   we can ignore the factor $x^{pq}+1$, so the polynomial begins
   $\sum_{j\geq 0} (-1)^j (x^{jq+(q-3)/2+2} - x^{jq+(q-3)/2})$.
   The coefficients appear in $\pm 1$ pairs but for $j=(p-1)/2$ the
   pair is $\pm(x^{m-1}-x^{m+1})$.  Thus the sum of coefficients up
   to that of $x^{m-1}$ is $\pm 1$ and this is an improper identity.
   
   To complete the proof, note that $t/2-1=0$ in the case $t=2$,
   which occurs only for $n=9$ and twice an odd prime.
\end{proof}
The case of prime $n$ was previously noted by Simone Costa~\cite{BurattiPC}.

It is likely that the presence of an identity implies that
there are two distinct realizable multisets with the same length,
but this is something that remains open.  It is plausible,
if unlikely, that the constraints on realizability of multisets
sometimes preclude the difference of two realizable multisets
ever being an identity.

\subsection{Generators }\label{generators}

In this section we record generators for the vector spaces of
identities.
All cases for $n\leq 37$ which are not mentioned
have dimension~0.

\newcommand{\ilist}[1]{\nicebreak\noindent $n=#1$\par\kern-2\parskip}

\medskip

\ilist{12}
\begin{verbatim}
   [1, -2, 1, 0, -1, 1]
\end{verbatim}

\ilist{15}
\begin{verbatim}
   [1, 0, -1, -1, -1, 0, 2]
   [0, 1,  0, -2, -1, 1, 1]
\end{verbatim}

\ilist{18}
\begin{verbatim}
   [1, 0, -2, 0, 1, 0, -1,  0, 1]
   [0, 1, -2, 1, 0, 0,  0, -1, 1]
\end{verbatim}

\ilist{20}
\begin{verbatim}
   [1, -2, 1, 0, -1, 2, -1, 0, 1, -1]
\end{verbatim}

\ilist{21}
\begin{verbatim}
   [1, 0, 0, -1, -2,  0, 1, 1, 1, -1]
   [0, 1, 0, -1, -1, -1, 1, 2, 0, -1]
   [0, 0, 1,  0, -2, -1, 1, 2, 1, -2]
\end{verbatim}

\ilist{24}
\begin{verbatim}
   [1, 0, 0, -2, 0, 0, 1, 0, -1,  0,  0, 1]
   [0, 1, 0, -2, 0, 1, 0, 0,  0, -1,  0, 1]
   [0, 0, 1, -2, 1, 0, 0, 0,  0,  0, -1, 1]
\end{verbatim}

\ilist{25}
\begin{verbatim}
   [1, -1, -1, 1, 0, -1, 1, 1, -1, 0, 1, -1]
\end{verbatim}

\ilist{27}
\begin{verbatim}
   [1, 0, 0, -1, -1, 0, 0, 1, 0, -1,  0,  0, 1]
   [0, 1, 0, -1, -1, 0, 1, 0, 0,  0, -1,  0, 1]
   [0, 0, 1, -1, -1, 1, 0, 0, 0,  0,  0, -1, 1]
\end{verbatim}

\ilist{28}
\begin{verbatim}
   [1, -2, 1, 0, -1, 2, -1, 0, 1, -2, 1, 0, -1, 1]
\end{verbatim}
   
\ilist{30}
\begin{verbatim}
   [1, 0, 0, 0, 0, 0, -2,  0, -1,  0, -1,  0, 2,  0, 1]
   [0, 1, 0, 0, 0, 0, -2,  1, -2,  0,  0, -1, 2,  0, 1]
   [0, 0, 1, 0, 0, 0, -1,  0, -2,  0,  0,  0, 1,  0, 1]
   [0, 0, 0, 1, 0, 0,  0, -2,  0, -1,  0,  1, 0,  1, 0]
   [0, 0, 0, 0, 1, 0, -1,  0, -1,  0,  0,  0, 1,  0, 0]
   [0, 0, 0, 0, 0, 1, -2,  2, -2,  1,  0, -1, 2, -2, 1]
\end{verbatim}

\ilist{33}
\begin{verbatim}
   [1, 0, 0, 0, 0, -1, -2, -1, 1, 3, 1, -2, -2, -1, 1, 2]
   [0, 1, 0, 0, 0, -1, -2, -1, 2, 2, 1, -1, -3, -1, 1, 2]
   [0, 0, 1, 0, 0, -1, -2,  0, 1, 2, 1, -1, -2, -2, 1, 2]
   [0, 0, 0, 1, 0, -1, -1, -1, 1, 2, 1, -1, -2, -1, 0, 2]
   [0, 0, 0, 0, 1,  0, -2, -1, 1, 2, 1, -1, -2, -1, 1, 1]
\end{verbatim}

\ilist{35}
\begin{verbatim}
   [1, 0, 0, 0, -1, -1, -1, -1,  0, 1, 2, 2, 1, 1,  0, -2, -2]
   [0, 1, 0, 0,  0, -2, -1,  0, -1, 1, 2, 1, 2, 1, -1, -1, -2]
   [0, 0, 1, 0, -1,  0, -1, -1,  1, 0, 0, 2, 1, 0,  0, -1, -1]
   [0, 0, 0, 1,  0, -2,  0,  1, -1, 0, 1, 0, 1, 1, -1, -1,  0]
\end{verbatim}

\ilist{36}
\begin{verbatim}
   [1, 0, 0, 0, 0, -2, 0, 0, 0, 0, 1, 0, -1,  0,  0,  0,  0, 1]
   [0, 1, 0, 0, 0, -2, 0, 0, 0, 1, 0, 0,  0, -1,  0,  0,  0, 1]
   [0, 0, 1, 0, 0, -2, 0, 0, 1, 0, 0, 0,  0,  0, -1,  0,  0, 1]
   [0, 0, 0, 1, 0, -2, 0, 1, 0, 0, 0, 0,  0,  0,  0, -1,  0, 1]
   [0, 0, 0, 0, 1, -2, 1, 0, 0, 0, 0, 0,  0,  0,  0,  0, -1, 1]
\end{verbatim}

\nicebreak
\section{Counting distinct lengths}\label{distinct}

Having verified that  the realizable multisets are the admissible multisets
for $n\leq 37$, our next task is to determine how many distinct lengths
occur for the admissible multisets.

One way is to compute accurate numerical approximations for the lengths,
sort them, then rigorously verify equality for those lengths which are
no further apart than rounding error can explain.
We carried this out up to $n=28$ but memory limits prevented us
from going further.  This led us to a better method.

For a multiset $M\in\calA_n$, let $\calL(M)$ be the set of all
multisets in $\calA_n$ that have the same length as $M$,
including~$M$ itself.
A multiset $M$ is \textit{minimal} if it is lexicographically least in~$\calL(M)$.
Since each set $\calL(M)$ has exactly one minimal element,
we have that the number of distinct lengths equals the number
of minimal admissible multisets.

The task is thus reduced to recognizing minimal multisets.  Recall that
admissible multisets $M,M'$ have the same length if and only if $M-M'$ is an identity.
So, if $M+A$ is an admissible multiset for some nonzero identity $A$ whose first
nonzero entry is negative, then $M$ is not minimal.
We will say that $A$ \textit{eliminates}~$M$.
If there is no such $A$ for which $M+A$ is an admissible multiset,
then $M$ is minimal.

The number of identities to test is reduced to a finite number by
noting that $M+A$ has at least one negative entry if some
subset of entries in $A$ has sum greater than~$n-1$.
However, in practice there are too many identities remaining.
For $n=30$ there are 1,552,732 identities and 78,356,395,953
admissible multisets; the combination is infeasible. 
For $n=36$ the situation is even worse: 214,302 identities
and 21,944,254,861,680 admissible multisets.
Fortunately we do not need to test so many identities.

For a multiset or identity $X$, and $2\leq d\leq m$, let
$\Sd(X)$ be the sum of the entries of~$X$ whose position
is divisible by~$d$.  Recall that the definition of admissibility
of a multiset~$M$ is that $\Sd(M)\leq n-d$ whenever~$d$
is a divisor of~$n$.

For identities $A=(a_1,\ldots,a_m)$ and $B=(b_1,\ldots,b_m)$
write $B\redto A$ if the following two conditions are satisfied.\\
(a) For $1\leq j\leq m$, either $a_j\geq 0$ or $a_j\geq b_j$.\\
(b) For each divisor $d$ of $n$, either $\Sd(A)\leq 0$ or $\Sd(A)\leq \Sd(B)$.

\begin{lemma}\label{redto}
  Let $A,B$ be identities with $B\redto A$.
  Then if $B$ eliminates admissible multiset~$M$, so does $A$.
\end{lemma}
\begin{proof}
  Let $M=[\ell_1,\ldots,\ell_m]$, $A=(a_1,\ldots,a_m)$ and
  $B=(b_1,\ldots,b_m)$.
  We are given that $M$ and $M+B$ are admissible multisets,
  and need to show that $M+A$ is also an admissible multiset.
  
  For $1\leq j\leq m$, if $a_j\geq 0$ then
  $\ell_j+a_j \geq \ell_j\geq 0$, whereas if $a_j\geq b_j$
  then $\ell_j+a_j \geq \ell_j+b_j\geq 0$.
  So $M+A$ is nonnegative, i.e., is a multiset. 
  
  For divisor $d$ of $n$, if 
  $\Sd(A)\leq 0$ then $\Sd(M+A)\leq \Sd(M)\leq n-d$,
  whereas if $\Sd(A)\leq\Sd(B)$ then
  $\Sd(M+A)\leq \Sd(M+B)\leq n-d$.
  So $M+A$ is admissible.
  This completes the proof.
\end{proof}

Lemma~\ref{redto} is surprisingly powerful.  Start with all
identities whose first nonzero entry is negative and such that
no subset of the entries sums to greater than $n-1$.
Then repeatedly remove identities $B$ from the set
if there is a different identity $A$ still in the set such that $B\redto A$.
At each stage, Lemma~\ref{redto} guarantees that the ability to
eliminate multisets is maintained.
For $n=30$, the number of required identities is reduced from
1,552,732 to~$65$.
The count for each $n$ is shown in the last column of
Table~\ref{table}.

\begin{table}[tp!]
\centering
\setlength{\tabcolsep}{0.8em}
\begin{tabular}{c|ccc|cc}
 $n$ & $\abs{\calM_n}$ & $\abs{\calR_n}=\abs{\calA_n}$ &
  distinct lengths & Dimen & Essential \\[0.3ex]
\hline
  3 & 1 & 1 & 1 &   &  \\
  4 & 4 & 3 & 3 &   &  \\
  5 & 5 & 5 & 5 &   &  \\
  6 & 21 & 17 & 17 &   &  \\
  7 & 28 & 28 & 28 &   &  \\
  8 & 120 & 105 & 105 &   &  \\
  9 & 165 & 161 & 161 &   &  \\
10 & 715 & 670 & 670 &   &  \\
11 & 1001 & 1001 & 1001 &   &  \\
12 & 4368 & 4129 & 2869 & 1 & 1\\
13 & 6188 & 6188 & 6188 &   &  \\
14 & 27132 & 26565 & 26565 &   &  \\
15 & 38760 & 38591 & 14502 & 2 & 4\\
16 & 170544 & 167898 & 167898 &   &  \\
17 & 245157 & 245157 & 245157 &   &  \\
18 & 1081575 & 1072730 & 445507 & 2 & 3\\
19 & 1562275 & 1562275 & 1562275 &   &  \\
20 & 6906900 & 6871780 & 6055315 & 1 & 1\\
21 & 10015005 & 10011302 & 2571120 & 3 & 7\\
22 & 44352165 & 44247137 & 44247137 &   &  \\
23 & 64512240 & 64512240 & 64512240 &   &  \\
24 & 286097760 & 285599304 & 65610820 & 3 & 6\\
25 & 417225900 & 417219530 & 362592230 & 1 & 1\\
26 & 1852482996 & 1850988412 & 1850988412 &   &  \\
27 & 2707475148 & 2707392498 & 591652989 & 3 & 6\\
28 & 12033222880 & 12026818454 & 11453679146 & 1 & 1\\
29 & 17620076360 & 17620076360 & 17620076360 &   &  \\
30 & 78378960360 & 78356395953 & 1511122441 & 6 & 65\\
31 & 114955808528 & 114955808528 & 114955808528 &   &  \\
32 & 511738760544 & 511647729284 & 511647729284 &   &  \\
33 & 751616304549 & 751614362180 & 67876359922 & 5 & 40\\
34 & 3348108992991 & 3347789809236 & 3347789809236 &   &  \\
35 & 4923689695575 & 4923688862065 & 1882352047787 & 4 & 32\\
36 & 21945588357420 & 21944254861680 & 1404030562068 & 5 & 17\\
37 & 32308782859535 & 32308782859535 & 32308782859535  &   & 
\end{tabular}
\caption{Counts of realizable multisets and the number of distinct lengths.
``Dimen'' is the dimension of the vector space of identities and
 ``Essential'' is the number of identities required in Section~\ref{distinct}.
 For readability, zeros in the last two columns are left blank.
 \label{table}}
\end{table}

\begin{table}[tp!]
\centering
\setlength{\tabcolsep}{0.8em}
\begin{tabular}{c|cc}
 $n$  & $\abs{\calM_n}$& $\abs{\calA_n}$ \\[0.3ex]
\hline
38  38 & 144079707346575 & 144074954225730\\
39  39 & 212327989773900 & 212327943155328\\
40  40 & 947309492837400 & 947290091984737\\
41  41 & 1397281501935165 & 1397281501935165\\
42  42 & 6236646703759395 & 6236574886430483\\
43  43 & 9206478467454345 & 9206478467454345\\
44  44 & 41107996877935680 & 41107708028136365\\
45  45 & 60727722660586800 & 60727721456103761\\
46  46 & 271250494550621040 & 271249413252489750\\
47  47 & 400978991944396320 & 400978991944396320\\
48  48 & 1791608261879217600 & 1791603906671596709\\
49  49 & 2650087220696342700 & 2650087220630545150\\
50  50 & 11844267374132633700 & 11844250906909678730
\end{tabular}
\caption{Counts of admissible multisets.
These have not been shown to be realizable.
 \label{table2}}
\end{table}

\section{Results}\label{results}

By elementary combinatorics, $\abs{\calM_n}=\binom{n+m-2}{m-1}$.
The size of $\calA_n$ has no formula that we know of, but
it is easy to compute for small~$n$.

The most expensive task was the verification that 
$\calR_n=\calA_n$ for $n\leq 37$, which took approximately
four years of cpu time.
By contrast, counting distinct lengths took only about 500 hours.

While the authors shared ideas, in the interest of establishing
independent reproducibility they did not share code, hardware, or even
programming languages.
All of the computations were completed independently by the two authors
except for the very expensive realization of admissible multisets for
$29\leq n\leq 37$.

The counts resulting from our computations are shown in Table~\ref{table}.
Additional values of $\abs{\calA_n}$, which took less than one minute
to compute, are given in Table~\ref{table2}. 
Note that these additional admissible multisets have not been tested for
realizability.

The average testing time per multiset generally grew at a slower rate than the
number of multisets, so the latter is the main indicator for how expensive
it would be to extend the computation to larger sizes.  We also observed
that realizability testing tended to be more difficult if $n$ is highly composite,
compared to prime or near-prime.

\section{OEIS sequences}

This paper extends the following entries in the Online Encyclopedia of
Integer Sequences.

\seqnum{A030077} ~Take $n$ equally spaced points on circle, connect them by a path with $n-1$ line segments; sequence gives number of distinct path lengths.

\seqnum{A352568}  ~Take $n$ equally spaced points on circle, connect them by a path with $n-1$ line segments; sequence gives number of distinct multisets of segment lengths.

\pagebreak

\end{document}